\documentclass[12pt]{article}
\usepackage{amsthm}
\usepackage{amsmath}
\usepackage{amsfonts}
\usepackage{graphicx}
\usepackage{latexsym}
\usepackage{amssymb}
\usepackage{lmodern}

\newcommand{\field}[1]{\mathbb{#1}}

\newcommand{\Z}{\field{Z}}

\newcommand{\F}{\field{F}}

\newcommand{\cC}{{\cal C}}

\newcommand{\cF}{{\cal F}}

\newcommand{\cS}{{\cal S}}

\newcommand{\cT}{{\cal T}}

\newcommand{\cR}{{\cal R}}

\newcommand{\cN}{{\cal N}}

\newcommand{\cG}{{\cal G}}

\newcommand{\mmod}{{\textup{mod}}}

\newtheorem{definition}{Definition}

\newtheorem{theorem}{Theorem}
\newtheorem{lemma}{Lemma}

\newtheorem{corollary}{Corollary}

\newtheorem{example}{Example}

\begin{document}

\bibliographystyle{plain}

\title{
\begin{center}
Mixed Steiner Triple Systems \\
with Shortest Length
\end{center}
}
\author{
{\sc Tuvi Etzion}\thanks{Department of Computer Science, Technion,
Haifa 3200003, Israel, e-mail: {\tt etzion@cs.technion.ac.il}.}}

\maketitle

\begin{abstract}
We prove that a 3-GDD of type $1^n k^1 \ell^1$, where $n= k \cdot \ell$, with minimum distance 3 exists
for every $k$ and $\ell$ such that $n = k \cdot \ell$, $k \equiv 1$ or $3~(\mmod ~ 6)$, and
$\ell \equiv 1$ or $3~(\mmod ~ 6)$. These designs are of the shortest possible length (smallest number of elements) for
given $k$ and $\ell$. Other constructions for such triple systems are also presented.
\end{abstract}

\vspace{0.5cm}



\newpage
\section{Introduction}
\label{sec:introduction}

A \emph{Steiner system of order $n$}, S$(t,k,n)$, is a pair $(\cN,B)$, where $\cN$ is an $n$-set
(whose elements are called \emph{points}) and $B$ is a collection
of $k$-subsets (called \emph{blocks}) of~$\cN$, such that each $t$-subset of~$\cN$ is contained
in exactly one block of $B$. A Steiner system can be represented by a binary code~$\cC$ whose codewords have length $n$ and
weight $k$. For each word $x$ of length $n$ and weight $t$, there is exactly one codeword $c \in \cC$ for
which $d(x,c)=k-t$, where $d(y,z)$ is the Hamming distance between the words $y$ and $z$.
As a code an S$(t,k,n)$ has minimum distance $2(k-t)+2$.

There are several generalization of Steiner systems, such as group divisible designs~\cite{Han75}, orthogonal arrays~\cite{HSS99},
generalized Steiner systems~\cite{Etz97}, etc.~. Some of the generalizations considers only the type of blocks in the design and
some also consider the minimum distance. As an example we can consider group divisible design (GDD) with blocks of size 3. This design
does not involve any restriction concerning the minimum distance. Generalized Steiner systems with codewords
(blocks) of size three is the same design, in which minimum distance 3 is required.

In the paper we denote by $\Z_r$ the set of integers $\{ 0,1,\ldots,r-1\}$ with addition modulo $r$.
Let $\Z^*_r$ denote the set of $r-1$ nonzero integers of $\Z_r$, i.e., $\Z^*_r = \Z_r \setminus \{ 0 \}$.

A \emph{mixed Steiner system} MS$(t,k,Q)$, over the mixed alphabet $Q=\Z_{q_1} \times \Z_{q_2} \times \cdots \times \Z_{q_n}$,
is a pair $(Q,\cC)$, where $\cC$ is a set of codewords (blocks) of weight $k$, over $Q$, such that for each word $x$ of weight $t$ over $Q$, there
exists exactly one codeword $c \in \cC$, such that $c$ \emph{covers} $x$, i.e., $d(x,c)=k-t$, and
the minimum Hamming distance of the code $\cC$ is $2(k-t)+1$. When all the $q_i$'s are equal the mixed Steiner system
is a generalized Steiner system and if the minimum distance is not considered the generalized Steiner system is a H-design
or a $(t,k)$-GDD. If not all the $q_i$'s are equal and the minimum distance is not considered, then the mixed Steiner system
is just a nonuniform $(t,k)$-GDD. A \emph{nonuniform $(t,k)$-GDD} of type $g_1^{r_1} g_2^{r_2} \cdots g_s^{r_s}$ consists of a collection of
$\sum_{i=0}^s r_i$ disjoint groups of $s$ sizes, $g_i$, $1 \leq i \leq s$, where the number of groups of size $g_i$ is $r_i$.
A $(t,k)$-GDD contains blocks with $k$ element, each one from a different group, such that each subset of $t$ elements
from $t$ different groups is contained in exactly one block. Mixed Steiner systems were defined in~\cite[Chapter 7]{Etz22}
and further discussed in~\cite{Etz25}.

Our goal in this paper is to consider \emph{mixed Steiner triple systems (MSTSs)} MS$(2,3,Q)$. It is not difficult to observe
that if $Q= \Z_2^n \times \Z_{k+1}$, then there is no difference between the mixed Steiner systems and the associated $(2,3)$-GDDs
(since only these$(2,3)$-GDDs will be discuss we will omit the parameters $(2,3)$) of type $1^n k^1$.
This case of GDDs was completely solved in~\cite{CHR92}. Moreover, the existence of GDDs of type $1^n k^r$, where $r > 1$ was
completely settled in~\cite{CCK95}. However, when $r >1$, the GDD is not necessarily a MSTS and there are parameters
where the GDD exists while an MSTS with the same parameters does not exists. In this paper we consider MSTS MS$(2,3,Q)$
for $Q= \Z_2^n \times \Z_{k+1} \times \Z_{\ell+1}$, where $k \geq \ell$.

The rest of the paper is organized as follows.
Section~\ref{sec:Preli} discusses some of the structures required for our constructions,
introduces some known results and especially those results which yield
MSTS MS$(2,3,Q)$ for $Q= \Z_2^n \times \Z_{k+1} \times \Z_{\ell+1}$.
Section~\ref{sec:conditions} examines some of the necessary conditions
for the existence of MSTSs and discusses some of the difference between them and the associated GDDs.
Section~\ref{sec:const_MSTS} presents a construction for MSTSs of the shortest possible
length (smallest $n$ given $k$ and $\ell$) as required by the necessary conditions.
A recursive construction of other MSTSs are presented in Section~\ref{sec:more}.
Conclusion and future research are discussed in Section~\ref{sec:conclude}

\section{Preliminaries}
\label{sec:Preli}

We start by a representation of a codeword in an MSTS MS$(2,3,\Z_2^n \times \Z_{k+1} \times \Z_{\ell+1})$. We will use two representations
The first $n$ coordinates are used for $\Z_2^n$, the $n$-th coordinate for $\Z_{k+1}$ and
the $(n+1)$-th coordinate for $\Z_{\ell+1}$.
One representation is to represent a triple
as $\{(i_1,\alpha_1),(i_2,\alpha_2),(i_3,\alpha_3)\}$ which means that the codeword has three nonzero entries at
positions $i_1$, $i_2$, and $i_3$. The nonzero value at $i_1$ is $\alpha_1$, the nonzero value at $i_2$ is $\alpha_2$, and
the nonzero at $i_3$ is $\alpha_3$. In the second representation for codewords of an MS$(2,3,\Z_2^n \times \Z_{k+1} \times \Z_{\ell+1})$, where
$n=k \cdot \ell$, the first $n$ coordinates are considered over $\Z_k \times \Z_\ell$ and a triple is represented
as $\{ [i_1,j_1],[i_2,j_2],[i_3,j_3]\}$ where a pair such as $[i_1,j_1]$ means that there is an \emph{one} in coordinate $[i_1,j_1]$.
A triple such as $\{ [i,j],(n,\alpha),(n+1,\beta)\}$ means that there is an \emph{one} in coordinate $[i,j]$, $\alpha \in \Z^*_{k+1}$
in the $n$-th coordinate and $\beta \in \Z^*_{\ell+1}$ in the $(n+1)$-th coordinate.

\begin{definition}
A \emph{Steiner triple system of order $v$}, STS$(v)$, is an S$(2,3,v)$. Such a system has $\frac{v(v-1)}{2}$ blocks.
\end{definition}

Mixed Steiner systems can be obtained from perfect mixed codes defined next.

\begin{definition}
\label{def:mixedPerfect}
An \emph{$e$-perfect mixed code} $\cC$ is a code over a mixed alphabet $Q$, where
$Q= \Z_{q_1} \times \Z_{q_2} \times \cdots \times \Z_{q_{n-1}} \times \Z_{q_n}$,
and covering radius $e$ (minimum distance $2e+1$) is a code whose codewords have length $n$ over the alphabet $Q$. For each word $x \in Q$ there exists
a unique codeword $c \in \cC$ such that $d(x,c) \leq e$.
\end{definition}

It is well-known that over an alphabet $\Z_2^n$ in such a perfect code, that
contains the all-zero codeword, the codewords
of weight $2e+1$ form a Steiner system S$(e,2e+1,n)$.  A straightforward generalization yields
the following lemma.

\begin{lemma}
\label{lem:fromMperfect}
If an $e$-perfect mixed code over $Q$ contains the all-zero codeword, then the codewords
of weight $2e+1$ form a mixed Steiner system MS$(e,2e+1,Q)$.
\end{lemma}

We continue with 1-perfect mixed  codes based on partitions of the nonzero element of $\F_{q^m}$ into subsets, each one a subspace
if the zero element is added. The following theorem was proved for example in~\cite{Zar51,Zar52} and also in~\cite[p. 180, Theorem 7.1]{Etz25}.

\begin{theorem}
\label{thm:mixed_partitions}
Let $\{\cS_1, \cS_2,\ldots,\cS_n \}$ be a partition of $\F^{q^m} \setminus \{ \bf0 \}$ into $n$ subsets such that
$\cT_i \triangleq \cS_i \cup \{ \bf0 \}$ is a subspace of dimension $k_i$ over $\F_q$. The code defined by
$$
\cC \triangleq \left\{ (c_1,c_2,\ldots,c_n) ~:~ c_i \in \cS_i \cup \{ 0 \},~ \sum_{i=1}^n c_i =0  \right\},
$$
where the sum is performed in $\F_{q^m}$,
is a 1-perfect mixed code over ${\cT_1 \times \cT_2  \times \cdots \times \cT_n}$ which is
isomorphic to $\F_{q^{k_1}} \times \F_{q^{k_2}} \times \dots \times \F_{q^{k_n}}$.
\end{theorem}

Perfect mixed codes were considered first by Sch\"{o}nheim~\cite{Sch70} and later by~\cite{HeSc71,HeSc72}
(see~\cite{Etz22} for more description, results, and references).
All the codes considered in these papers have covering radius 1.
The definition of perfect mixed codes, Lemma~\ref{lem:fromMperfect}, and their construction
in Theorem~\ref{thm:mixed_partitions} implies a very simple way to
construct a MSTS MS$(2,3,Q)$ by partitioning the space into disjoint subspaces (omitting the all-zero vector). Such partitions
were extensively studied in the literature, e.g.,~\cite{ESSSV07} and they are associated with byte-correcting codes~\cite{HoPa72}
as explained in~\cite{Etz98}.

The last definition is for one-factorization and near-one factorization of the complete graph $K_n$ on the set of points $\Z_n$.
These concepts are extensively studied in~\cite{Wal97}.

\begin{definition}
A \emph{one-factor} of $K_n$ on the point set $\Z_n$, where $n$ is even, is a partition of $\Z_n$ into $\frac{n}{2}$ disjoint pairs.
A \emph{one-factorization} of $K_n$ on the point set $\Z_n$, where $n$ is even, is a partition of all the $\binom{n}{2}$ pairs
of $\Z_n$ into $n-1$ disjoint one-factors.
\end{definition}

\begin{definition}
A \emph{near-one-factor} of $K_n$ on the point set $\Z_n$, where $n$ is odd, is a partition of $\Z_n$ into $\frac{n-1}{2}$ disjoint pairs
and one isolated point.
A \emph{near-one-factorization} of $K_n$ on the point set $\Z_n$, where $n$ is odd, is a partition of all the $\binom{n}{2}$ pairs
of $\Z_n$ into $n$ disjoint near-one-factors. If the near-one-factorization is $\{ F_0,F_1,\cdots,F_{n-1} \}$ we assume that the isolated
point in $F_i$, $0 \leq i \leq n-1$, is $i \in \Z_n$.
\end{definition}

\section{Necessary Conditions}
\label{sec:conditions}

Necessary and sufficient conditions for the existence of a GDD of type $1^n k^2$ were
presented in~\cite{CCK95}. These conditions are also necessary conditions for the existence of an
MS$(2,3, \Z_2^n \times \Z^2_{k+1})$, but these conditions are not enough due to minimum distance
required for the mixed Steiner system. Moreover, these conditions are not enough for a GDD of type $1^n k^1 \ell^1$ when $k \neq \ell$,

The first lemma presents five necessary conditions for the existence of an MS$(2,3, \Z_2^n \times \Z_{k+1} \times \Z_{\ell+1})$.

\begin{lemma}
\label{lem:necessaryMSTS}
If there exists an \textup{MS}$(2,3, \Z_2^n \times \Z_{k+1} \times \Z_{\ell+1})$ then
\begin{enumerate}
\item[\textup{(1)}] $n - k \equiv 0~(\mmod~2)$.

\item[\textup{(2)}] $n - \ell \equiv 0~(\mmod~2)$.

\item[\textup{(3)}] $k \equiv \ell \equiv 1~(\mmod ~2)$.

\item[\textup{(4)}] $n \geq k \cdot \ell$.

\item[\textup{(5)}] $k \cdot \ell + (k + \ell)n + \binom{n}{2} \equiv 0~(\mmod~3)$.
\end{enumerate}
\end{lemma}
\begin{proof}
$~$
\begin{enumerate}
\item[(1)] There are $k+n$ nonzero points in $\Z_2^n \times \Z_{k+1}$.
Each point of $\Z^*_{\ell+1}$ is paired with each nonzero element of $\Z_2^n \times \Z_{k+1}$ in a codeword.
Each such point of $\Z^*_{\ell+1}$ is contained in codewords with disjoint nonzero pairs of $\Z_2^n \times \Z_{k+1}$. Hence,
$k+n$ is even, i.e., $n - k \equiv 0~(\mmod~2)$.

\item[(2)] The proof is identical to the one of (1).

\item[\textup{(3)}] Similarly, each word of weight one of $\Z_2^n$
must appear with each other point of $\Z_2^n \times \Z_{k+1} \times \Z_{\ell+1}$ in codewords of weight three
of $\Z_2^n \times \Z_{k+1} \times \Z_{\ell+1}$ and hence $k+\ell+n-1$ must be even. (1) and (2) imply that
$k+\ell$ is even and hence $n$ is odd. Therefore, by (1) and (2) also $k$ and $\ell$ are odds.

\item[(4)] There are $k \cdot \ell$ pairs of the form $\{(n,i),(n+1,j)\}$, where $i \in \Z^*_{k+1}$ and $j \in \Z^*_{\ell+1}$.
Each such pair is covered by a unique codeword of the form $\{(m,1),(n,i),(n+1,j)\}$, where $m \in \Z_n$. Hence, there are $k \cdot \ell$ codewords
of this form.
For each two distinct codewords $X=\{(m,1),(n,i_1),(n+1,j_1)\}$ and $Y=\{(p,1),(n,i_2),(n+1,j_2)\}$ we have that
$d(X,Y) \geq 3$ and hence $m \neq p$ and since there are $k \cdot \ell$ codewords of this form, it follows that $n \geq k \cdot \ell$.

\item[(5)] There are $k \cdot \ell$ pairs of the form $\{(n,i),(n+1,j)\}$, where $i \in \Z^*_{k+1}$ and $j \in \Z^*_{\ell+1}$.
There are $k \cdot n$ pairs of the form $\{ (m,1),(n,i) \}$, where $m \in \Z_n$ and $i \in \Z^*_{k+1}$ and similarly
there are $\ell \cdot n$ pairs of the form $\{ (m,1),(n+1,j) \}$, where $m \in \Z_n$ and $j \in \Z^*_{\ell+1}$.
Finally, there are $\binom{n}{2}$ pairs of the form $\{ (m,1),(p,1) \}$, where $m,p \in \Z_n$.
Each such pair is covered by exactly one codeword and each codeword contains three distinct pairs
and therefore $k \cdot \ell + (k + \ell)n + \binom{n}{2} \equiv 0~(\mmod~3)$.
\end{enumerate}
\end{proof}

As a consequence from the proof of Lemma~\ref{lem:necessaryMSTS}(5) we have an enumeration on the number
of codewords in an MS$(2,3,\Z_2^n \times \Z_{k+1} \times \Z_{\ell+1})$.
\begin{corollary}
\label{cor:numCodeW}
The number of codewords in an \textup{MS}$(2,3,\Z_2^n \times \Z_{k+1} \times \Z_{\ell+1})$ is $(k \cdot \ell + (k + \ell)n + \binom{n}{2})/3$.
\end{corollary}

The necessary conditions of Lemma~\ref{lem:necessaryMSTS} (1), (2), (3), and (5), are also necessary conditions for the associated GDD.
The only condition which takes the minimum distance into account is (4). We are first interested which values of $k$ and $\ell$ can yield
a code which meets this condition with equality, i.e., $n=k \cdot \ell$.
By Lemma~\ref{lem:necessaryMSTS} we have that $k$, $\ell$, and $n$ must be odd integers.
It is not difficult to verify
that all odd values of $k$ and $\ell$ are possible except for $k \equiv \ell \equiv 5 ~ (\mmod ~ 6)$.
Next, we look on the possible values of $k$, $\ell$, and $n$ modulo 6 for which we can have an
MS$(2,3, \Z_2^n \times \Z_{k+1} \times \Z_{\ell+1})$. Also in this case
$k$, $\ell$, and $n$ must be odd and when $k \equiv \ell \equiv 5 ~ (\mmod ~ 6)$, there is no $n$ for which such an MSTS exists.
If $k \equiv \ell \equiv 1 ~ (\mmod ~ 6)$, then $n \equiv 1$ or $5~(\mmod~6)$.
If $k \equiv \ell \equiv 3 ~ (\mmod ~ 6)$, then $n \equiv 1$ or $3~(\mmod~6)$.
If $k \equiv 1 ~ (\mmod ~ 6)$ and $\ell \equiv 3 ~ (\mmod ~ 6)$, then $n \equiv 3$ or $5~(\mmod~6)$.
If $k \equiv 1 ~ (\mmod ~ 6)$ and $\ell \equiv 5 ~ (\mmod ~ 6)$, then $n \equiv 5~(\mmod~6)$.
If $k \equiv 3 ~ (\mmod ~ 6)$ and $\ell \equiv 5 ~ (\mmod ~ 6)$, then $n \equiv 3~(\mmod~6)$.
The same is true when $k$ and $\ell$ switch values modulo 6.

Using Theorem~\ref{thm:mixed_partitions} we can construct MSTSs which attain the condition (4) of Lemma~\ref{lem:necessaryMSTS}.
For a given $n$ we take two integers $k$ and $\ell$ such that $k+\ell=n$. There exist two disjoint subspaces of~$\F_{2^n}$ with dimension
$k$ and $\ell$. Therefore, by Theorem~\ref{thm:mixed_partitions} and Lemma~\ref{lem:fromMperfect} we have the following simple consequence.

\begin{lemma}
If $n$, $k$, and $\ell$ are positive integers such that $n=k+\ell$, then there exist an \textup{MS}$(2,3,\Z_2^m \times \Z_{2^k} \times \Z_{2^\ell})$,
where $m=2^n - 2^k - 2^\ell +1 =(2^k-1)(2^\ell-1)$, which attains Lemma~\textup{\ref{lem:necessaryMSTS}(4)}.
\end{lemma}

If $n >k+\ell$ then using Theorem~\ref{thm:mixed_partitions} we obtain an \textup{MS}$(2,3,\Z_2^m \times \Z_{2^k} \times \Z_{2^\ell})$
which does not attain the bound of Lemma~\ref{lem:necessaryMSTS}(4) with equality, i.e., $m > (2^k-1)(2^\ell-1)$.

\section{Construction of Short Length MSTSs}
\label{sec:const_MSTS}

This section is devoted to the main result of this paper, a construction for MSTSs of the form
MS$(2,3,\Z_2^n \times \Z_{k+1} \times \Z_{\ell+1})$ with the shortest length, i.e.,
MSTSs which attain the necessary condition of Lemma~\ref{lem:necessaryMSTS}(4) with equality.
In other words, we define a construction for MS$(2,3,\Z_2^n \times \Z_{k+1} \times \Z_{\ell+1})$
for which $n= k \cdot \ell$. The construction will be combined of three
parts. In the first part we construct the codewords which cover all the pairs of the form $\{ (n,i),(n+1,j)\}$, where $i \in \Z^*_{k+1}$
and $j \in \Z^*_{\ell+1}$. In this part also some pairs of the form $\{[x,y],(n,i)\}$, where $x \in \Z_k$, $y \in \Z_\ell$, and
$i \in \Z^*_{k+1}$ and of the form $\{[x,y],(n+1,j)\}$, where $x \in \Z_k$, $y \in \Z_\ell$, and $j \in \Z^*_{\ell+1}$ will be covered.
In the second part we construct codewords which cover the other pairs of this form which are not covered in the first part.
There are pairs of the form $\{[x_1,y_1],[x_2,y_2]\}$, where $x_1,x_2 \in \Z_k$ and $y_1,y_2 \in \Z_\ell$ which will be covered
in this part. The pairs of this form which are not covered in the second part will be covered by codewords
which are constructed in the third part.

\noindent
{\bf Part I:}

It is trivial to cover the pairs of the form $\{ (n,i),(n+1,j)\}$. Moreover, there is a unique way to define the codewords which cover these triples
avoiding minimum distance less than 3.
The only task is to define these codewords in a way that it will be simple to define the rest of the codeword in Part II and Part III.
Hence, we define the first set of codewords with triples as follows:
$$
C_1 \triangleq \{  [i-1,j-1],(n,i),(n+1,j) ~:~ i \in \Z^*_{k+1},~ j \in \Z^*_{\ell+1} \}~.
$$

\noindent
{\bf Part II:}

The pairs of the form $\{[x,y],(n,i)\}$, where $x \in \Z_k$, $y \in \Z_\ell$, and $i \in \Z^*_{k+1}$ which are covered by the codewords
of $C_1$ are all those for which $x=i-1$. Similarly,
the pairs of the form $\{[x,y],(n+1,j)\}$, where $x \in \Z_k$, $y \in \Z_\ell$, $j \in \Z^*_{\ell+1}$ which are covered by the codewords
of $C_1$ are all those for which $y=j-1$. The rest of the pairs of these two forms will be covered by using two near-one-factorizations.

The first one on $\Z_k$,
$$
\cF \triangleq \{ F_0,F_1,\ldots,F_{k-1}  \},
$$
where in $F_i$, $0 \leq i \leq k-1$, the point $i \in \Z_k$ is isolated.

The second one on $\Z_\ell$,
$$
\cG \triangleq \{ G_0,G_1,\ldots,G_{\ell-1}  \},
$$
where in $G_i$, $0 \leq i \leq \ell -1$, the point $i \in \Z_\ell$ is isolated.

Now, we define two set of codewords in this part.
$$
C_2 \triangleq \{ \{[x,j],[y,j],(n,i)\} ~: i \in \Z^*_{k+1}, ~ j \in \Z_\ell, ~ \{x,y\} \in F_{i-1} \}
$$
$$
C_3 \triangleq \{ \{[i,x],[i,y],(n+1,j)\} ~: j \in \Z^*_{\ell+1}, ~ i \in \Z_k, ~ \{ x,y \} \in G_{j-1} \}
$$

A pair of the form $\{ [x,y],(n,i)\}$ where $x \in \Z_k$, $y \in \Z_\ell$, and $i \in \Z^*_{k+1}$ is covered as follows:
\begin{enumerate}
\item If $x=i-1$ then this pair is covered by the triple $\{[x,y],(n,i),(n+1,y+1)]\}$ defined in $C_1$.

\item If $x \neq i-1$, then there is a unique $z \in \Z_{k +1}$ such that $\{ x,z \} \in F_{i-1}$ and
this pair is covered by the triple $\{[x,y],[z,y],(n,i)]\}$ defined in $C_2$.
\end{enumerate}

A pair of the form $\{ [x,y],(n+1,j)\}$ where $x\in \Z_k$, $y \in \Z_\ell$, and $j \in \Z^*_{\ell+1}$ is covered in a similar way by codewords of $C_1$
and codewords of $C_3$.

\noindent
{\bf Part III:}

We now want to cover pairs from $\Z_k \times \Z_\ell$ which are not covered in Part II.
The pairs of $\Z_k \times \Z_\ell$ which are covered by the codewords of $C_2$ are all the pairs of the form
$\{ [x,j],[y,j]\}$, where $j \in \Z_\ell$ and $x,y \in \Z_k$ since $\{x,y\} \in F_r$ for some $0 \leq r \leq k -1$.
The pairs of $\Z_k \times \Z_\ell$ which are covered by the codewords of $C_3$ are all the pairs of the form
$\{ [i,x],[i,y]\}$, where $i \in \Z_k$ and $x,y \in \Z_\ell$ since $\{x,y\} \in G_r$ for some $0 \leq r \leq \ell -1$.
Hence, the pairs which remained to be covered
are all those of the form $\{[i_1,j_1],[i_2,j_2]\}$, where $i_1,i_2 \in \Z_k$, $j_1,j_2 \in \Z_\ell$, $i_1 \neq i_2$, and $j_1 \neq j_2$.
To cover these pairs we construct a set of codeword $C_4$ as follows.
Let $\cS_1$ be an STS$(k)$ on a set of point $\Z_k$ and $\cS_2$ be an STS$(\ell)$ on a set of points $\Z_\ell$. For each two triples
$\{ i_1,i_2,i_3\} \in \cS_1$ and $\{ j_1,j_2,j_3\} \in \cS_2$ we form the following six triples in $C_4$.
$$
\{ [i_1,j_1],[i_2,j_2],[i_3,j_3] \}, ~~~ \{ [i_1,j_1],[i_2,j_3],[i_3,j_2] \}, ~~~ \{ [i_1,j_2],[i_2,j_1],[i_3,j_3] \} ~,
$$
$$
\{ [i_1,j_2],[i_2,j_3],[i_3,j_1] \}, ~~~ \{ [i_1,j_3],[i_2,j_1],[i_3,j_2] \}, ~~~ \{ [i_1,j_3],[i_2,j_2],[i_3,j_1] \} ~.
$$

There are exactly 18 distinct pairs of the form $\{[x_1,y_1],[x_2,y_2]\}$ where $x_1,x_2 \in \{i_1,i_2,i_3\}$ and $y_1,y_2 \in \{j_1,j_2,j_3\}$,
$x_1 \neq x_2$, and $y_1 \neq y_2$. Each codeword of $C_4$ has three pairs and hence these 18 distinct pairs are contained in these
six codewords.

\begin{lemma}
Each word of weight two in $\Z_2^n \times \Z_{k+1} \times \Z_{\ell+1}$ is covered by exactly one codeword of $C_1 \cup C_2 \cup C_3 \cup C_4$.
\end{lemma}
\begin{proof}
The discussion on the three parts of the construction implies that each pair is covered by at least one codeword.
To complete the proof it is sufficient to show that the number of codewords is exactly $(k \cdot \ell + (k + \ell)n + \binom{n}{2})/3$
as indicated in Corollary~\ref{cor:numCodeW}.

In $C_1$ there are $k \cdot \ell$ codewords.

In $C_2$ there are $\frac{k \cdot \ell (k-1)}{2}$ codewords and similarly in $C_3$ there are $\frac{k \cdot \ell (\ell-1)}{2}$ codewords.

The STS$(k)$ $\cS_1$ has $\frac{k(k-1)}{6}$ codewords and the STS$(\ell)$ has
$\frac{\ell(\ell-1)}{6}$ codeword. Therefore, $C_4$ has $6 \cdot \frac{k(k-1)}{6} \cdot \frac{\ell(\ell-1)}{6} = \frac{k(k-1)\ell(\ell-1)}{6}$
codewords (recall that $n =k \cdot \ell$).

Therefore, since the $C_i$'s are disjoint, it follows that $C_1 \cup C_2 \cup C_3 \cup C_4$ contains
$$
k \cdot \ell + \frac{k \cdot \ell (k-1)}{2} + \frac{k \cdot \ell (\ell-1)}{2} + \frac{k(k-1)\ell(\ell-1)}{6} =
\frac{k \cdot \ell + (k + \ell)n + \binom{n}{2}}{3}
$$
codewords as required.
\end{proof}

\begin{corollary}
If $k \equiv 1$ or $3~(\mmod~6)$ and $\ell \equiv 1$ or $3~(\mmod~6)$, then there exists an \textup{MS}$(2,3,\Z_2^n \times \Z_{k+1} \times \Z_{\ell+1})$,
where $n = k \cdot \ell$.
\end{corollary}

When $k$ or $\ell$ is congruent to 5 modulo 6 the codewords of $C_4$ are designed in a different way as demonstrated in the following example.
\begin{example}
For $k=5$ and $\ell=3$, $C_1$, $C_2$, $C_3$, and $C'_4$ (to replace $C_4$), which yield an
\textup{MS}$(2,3,\Z_2^n \times \Z_6 \times \Z_4)$ are as follows.

$$
C_1=
\begin{array}{ccc}
\{[0,0],(15,1),(16,1)\} & \{[0,1],(15,1),(16,2)\} & \{[0,2],(15,1),(16,3)\} \\
\{[1,0],(15,2),(16,1)\} & \{[1,1],(15,2),(16,2)\} & \{[1,2],(15,2),(16,3)\} \\
\{[2,0],(15,3),(16,1)\} & \{[2,1],(15,3),(16,2)\} & \{[2,2],(15,3),(16,3)\} \\
\{[3,0],(15,4),(16,1)\} & \{[3,1],(15,4),(16,2)\} & \{[3,2],(15,4),(16,3)\} \\
\{[4,0],(15,5),(16,1)\} & \{[4,1],(15,5),(16,2)\} & \{[4,2],(15,5),(16,3)\} \\
\end{array}
$$

$$
C_2=
\begin{array}{ccc}
\{[1,0],[4,0],(15,1)\} & \{[2,0],[3,0],(15,1)\} & \{[1,1],[4,1],(15,1)\} \\
\{[2,1],[3,1],(15,1)\} & \{[1,2],[4,2],(15,1)\} & \{[2,2],[3,2],(15,1)\} \\
\{[0,0],[2,0],(15,2)\} & \{[3,0],[4,0],(15,2)\} & \{[0,1],[2,1],(15,2)\} \\
\{[3,1],[4,1],(15,2)\} & \{[0,2],[2,2],(15,2)\} & \{[3,2],[4,2],(15,2)\} \\
\{[0,0],[4,0],(15,2)\} & \{[1,0],[3,0],(15,2)\} & \{[0,1],[4,1],(15,2)\} \\
\{[1,1],[3,1],(15,2)\} & \{[0,2],[4,2],(15,2)\} & \{[1,2],[3,2],(15,2)\} \\
\{[0,0],[1,0],(15,2)\} & \{[2,0],[4,0],(15,2)\} & \{[0,1],[1,1],(15,2)\} \\
\{[2,1],[4,1],(15,2)\} & \{[0,2],[1,2],(15,2)\} & \{[2,2],[4,2],(15,2)\} \\
\{[0,0],[3,0],(15,2)\} & \{[1,0],[2,0],(15,2)\} & \{[0,1],[3,1],(15,2)\} \\
\{[1,1],[2,1],(15,2)\} & \{[0,2],[3,2],(15,2)\} & \{[1,2],[2,2],(15,2)\} \\
\end{array}
$$

$$
C_3=
\begin{array}{ccc}
\{[0,1],[0,2],(16,1)\} & \{[1,1],[1,2],(16,1)\} & \{[2,1],[2,2],(16,1)\} \\
\{[3,1],[3,2],(16,1)\} & \{[4,1],[4,2],(16,1)\} & \\
\{[0,0],[0,2],(16,2)\} & \{[1,0],[1,2],(16,2)\} & \{[2,0],[2,2],(16,2)\} \\
\{[3,0],[3,2],(16,2)\} & \{[4,0],[4,2],(16,2)\} & \\
\{[0,0],[0,1],(16,3)\} & \{[1,0],[1,1],(16,3)\} & \{[2,0],[2,1],(16,3)\} \\
\{[3,0],[3,1],(16,3)\} & \{[4,0],[4,1],(16,3)\} & \\
\end{array}
$$

$$
C'_4=
\begin{array}{cccc}
\{[0,0],[1,1],[2,2]\} & \{[0,0],[3,1],[4,2]\} & \{[0,0],[1,2],[2,1]\} & \{[0,0],[3,2],[4,1]\} \\
\{[0,1],[1,2],[3,0]\} & \{[0,1],[2,0],[4,2]\} & \{[0,1],[1,0],[3,2]\} & \{[0,1],[2,2],[4,0]\} \\
\{[0,2],[1,1],[4,0]\} & \{[0,2],[2,0],[3,1]\} & \{[0,2],[1,0],[4,1]\} & \{[0,2],[2,1],[3,0]\} \\
\{[1,0],[2,1],[4,2]\} & \{[1,0],[2,2],[3,1]\} & \{[1,1],[2,0],[3,2]\} & \{[1,1],[3,0],[4,2]\} \\
\{[1,2],[2,0],[4,1]\} & \{[1,2],[3,1],[4,0]\} & \{[2,1],[3,2],[4,0]\} & \{[2,2],[3,0],[4,1]\} \\
\end{array}
$$
\end{example}

\section{More Constructions for MSTSs}
\label{sec:more}

In this section we will construct MSTSs whose length is not the shortest one.
For this construction another structure is requires.

\begin{definition}
An \emph{$(n,r)$-pairs-triples design}, $\{ \cT_1,\cT_2,\ldots,\cT_r, \cR \}$, where $n$ is
even and $r$ is odd consists of $r+1$ sets. The $r$ sets $\cT_1,\cT_2,\ldots,\cT_r$ are disjoint one-factors of $K_n$
on the set of points~$\Z_n$, and one set $\cR$ of triples on the set of points $\Z_n$, such that each pair of $\Z_n$ is
contained either in one of the $\cT_i$'s or in a triple of $\cR$.
\end{definition}

\begin{theorem}
\label{thm:PT_GDD}
An $(n,r)$-pairs-triples design exists if and only if there exists a GDD of type $1^n r^1$, i.e., an \textup{MS}$(2,3,\Z_n \times \Z_{r+1})$.
\end{theorem}
\begin{proof}
Assume first that there exists an $(n,r)$-pairs-triples design on the points of $\Z_n$.
Let $\cT_1,\cT_2,\ldots,\cT_r$ be the set of $r$ disjoint one-factors and $\cR$ its set of triples.
The following sets of blocks
$$
\{ \{ (x,1),(y,1),(n,i) \} ~:~ \{x,y\} \in \cT_i, ~ 1 \leq i \leq r \}
$$
and
$$
\{ \{ (x,1),(y,1),(z,1) \} ~:~ \{x,y,z\} \in\cR \}
$$
is an MS$(2,3,\Z_n \times \Z_{r+1})$.

Assume now that $\cS$ is an MS$(2,3,\Z_n \times \Z_{r+1})$. Define the following $r$ sets
$$
\cT_i \triangleq \{ \{x,y\} ~:~ \{ (x,1),(y,1),(n,i) \} \in \cS \}, ~~ 1 \leq i \leq r
$$
and the set
$$
\cR \triangleq \{ \{ x,y,z \} ~:~ \{(x,1),(y,1),(z,1) \} \in \cS \}~.
$$
It is easy to verify that $\{ \cT_1,\cT_2,\ldots,\cT_r , \cR\}$ is an $(n,r)$-pairs-triples design.
\end{proof}

Note, that when $r=1$ an MS$(2,3,\Z_n \times \Z_{r+1})$ is just an S$(2,3,n+1)$.
For $r=n-1$ an $(n,r)$-pairs-triples design is just an one-factorization over $\Z_n$.
An MS$(2,3,\Z_n \times \Z_{r+1})$ is equivalent to a GDD of type $1^n r^1$ and since it is known that
the necessary conditions for the existence of a GDD of type $1^n r^1$ are also sufficient~\cite{CHR92}, the
following lemma is a consequence from Theorem~\ref{thm:PT_GDD}.

\begin{lemma}
An $(n,r)$-pairs-triples design exists for each $n$ divisible by 6 and odd $r$, $1 \leq r \leq n-1$ and
for each $n \equiv 2$ or $4~(\mmod ~6)$ and $r \equiv n-1 ~(\mmod ~6)$, $1 \leq r \leq n-1$.
\end{lemma}

We will now present a recursive construction for MSTSs.
Assume $\cS$ is an MS$(2,3,\Z_2^n \times \Z_{k+1} \times \Z_{\ell+1})$, where $k$, $\ell$, and $n$ are odd integers.
Assume further that there exists an $(m,r)$-pairs-triples design with
$\cT_1,\cT_2,\ldots,\cT_r$ as the set of $r$ disjoint one-factors and $\cR$ as the set of triples.
Assume further that $m > n+k+\ell$ and $r=n+k+\ell$. We construct a new system $\cS'$ on the points of
$\Z_2^n \times \Z_{k+1} \times \Z_{\ell+1} \times \Z_2^m$ as follows:
\begin{enumerate}
\item[(1)] All the blocks of $\cS$ are also blocks of $\cS'$ (same codewords where \emph{zeros} are assigned to the new $m$ coordinates).

\item[(2)] We add to $\cS'$ the set of blocks $\{ \{ (i,1),(n+1+\alpha,1),(n+1+\beta,1)\} ~:~ i \in \Z_n, ~ \{\alpha,\beta\} \in \cT_{i+1} \}$.

\item[(3)] We add to $\cS'$ the set of blocks $\{ \{ (n,j),(n+1+\alpha,1),(n+1+\beta,1)\} ~:~ j \in \Z^*_{k+1}, ~ \{\alpha,\beta\} \in \cT_{n+j} \}$.

\item[(4)] We add to $\cS'$ the set of blocks
$\{ \{ (n+1,j),(n+1+\alpha,1),(n+1+\beta,1)\} ~:~ j \in \Z^*_{\ell+1}, ~ \{\alpha,\beta\} \in \cT_{n+k+j} \}$.

\item[(5)] We add to $\cS'$ the set of blocks $\{ \{ (n+1+\alpha,1),(n+1+\beta,1),(n+1+\gamma)\} ~:~ \{\alpha,\beta,\gamma\} \in \cR \}$.
\end{enumerate}

\begin{theorem}
The constructed system $\cS'$ is an \textup{MS}$(2,3,\Z_2^n \times \Z_{k+1} \times \Z_{\ell+1} \times \Z_2^m)$.
\end{theorem}
\begin{proof}
Each word of weight two where the nonzero elements are in the first $n+2$ positions is covered exactly once by codewords defined in (1).
The codewords defined in (2), (3), (4), and (5) do not cover these words.

Each word of weight two where one nonzero element is in position $i$, $0 \leq i \leq n-1$ and one nonzero element is in the last $m$ positions
is covered exactly once by codewords defined in (2).

Each word of weight two where one nonzero element is in position $n$ and one nonzero element is in the last $m$ positions
is covered exactly once by codewords defined in (3).

Each word of weight two where one nonzero element is in position $n+1$ and one nonzero element is in the last $m$ positions
is covered exactly once by codewords defined in (4).

All the pairs in the last $m$ coordinates are covered exactly once by codewords defined in (2), (3), and (4), as defined by the
$(m,r)$-pairs-triples design.
\end{proof}

\section{Conclusion and Future research}
\label{sec:conclude}

We have analyzed mixed Steiner triple systems
MS$(2,3,\Z_2^n \times \Z_{k+1} \times \Z_{\ell+1})$ for the case where exactly in two positions the alphabet is not binary.
A construction for codes of shortest length, $n= k \cdot \ell$, where $k \equiv 1$ or $3~(\mmod ~ 6)$, and
$\ell \equiv 1$ or $3~(\mmod ~ 6)$ was presented. A general recursive construction was also presented.

Many open problems on MSTSs and associated GDDs remain for future research as well as
problems for other mixed Steiner systems. A few of these problems which are directly
connected to our exposition are as follows.
\begin{enumerate}
\item Are the necessary conditions of Lemma~\ref{lem:necessaryMSTS} are also sufficient for all odd $k$, $\ell$, and $n$?
If yes, we would like to see a proof. If no, we would like to see a proof of values for which the necessary conditions are not sufficient
and also more constructions for the cases when they are sufficient, especially where either $k$ or $\ell$ is congruent to 5 modulo 6.

\item What is the shortest length of a GDD of type $1^n k^1 \ell^1$, where $k \neq \ell$?

\item An obvious necessary condition for the existence of an MS$(2,\kappa,\Z_2^n \times \Z_{k+1} \times \Z_{\ell+1})$, where $\kappa \geq 3$, is that
$n \geq (\kappa -2) k \cdot \ell$. Are there mixed Steiner systems that attains this bound for each $\kappa$?
\end{enumerate}


%

\end{document}